\documentclass[12pt,a4paper]{article}
\usepackage[active]{srcltx}
\usepackage{amsfonts}
\usepackage{amsmath}
\usepackage{float}
\usepackage{cleveref}
\usepackage{graphicx,epsfig}
\usepackage{amssymb}
\usepackage{theorem,subcaption,mathtools}
\usepackage[usenames,dvipsnames]{color}
\usepackage{colortbl}
\definecolor{gray73}{RGB}{186,186,186}
\usepackage{graphicx}
\usepackage{tikz}
\usetikzlibrary{matrix,arrows}
\usepackage[american]{babel}
\usepackage{verbatim}
\usepackage{geometry}
\usepackage{lipsum}

\theoremstyle{plain}
\newtheorem{theorem}{Theorem}[section]

\newtheorem{lemma}[theorem]{Lemma}

\theoremstyle{definition}
\newtheorem{definition}[theorem]{Definition}

\newtheorem{proposition}[theorem]{Proposition}
\newtheorem{example}[theorem]{Example}
\newtheorem{algorithm}[theorem]{Algorithm}

\newtheorem{remark}[theorem]{Remark}

\newenvironment{proof}[1][Proof]{\noindent\textbf{#1:} }{\ \rule{0.5em}{0.5em}}

\numberwithin{equation}{section}

 at12pt
\newcommand{\sign}{\operatorname*{sign}}

\newcommand{\vertiii}[1]{{\left\vert\kern-0.25ex\left\vert\kern-0.25ex\left\vert #1
    \right\vert\kern-0.25ex\right\vert\kern-0.25ex\right\vert}}

\usepackage{sectsty}
\sectionfont{\large}
\subsectionfont{\normalsize}

\DeclareMathOperator{\prox}{prox}

\DeclareMathOperator*{\argmin}{arg\,min}

\DeclareFontFamily{U}{rcjhbltx}{}
\DeclareFontShape{U}{rcjhbltx}{m}{n}{<->rcjhbltx}{}
\DeclareSymbolFont{hebrewletters}{U}{rcjhbltx}{m}{n}

\let\aleph\relax\let\beth\relax
\let\gimel\relax\let\daleth\relax

\DeclareMathSymbol{\aleph}{\mathord}{hebrewletters}{39}
\DeclareMathSymbol{\beth}{\mathord}{hebrewletters}{98}
\DeclareMathSymbol{\gimel}{\mathord}{hebrewletters}{103}
\DeclareMathSymbol{\daleth}{\mathord}{hebrewletters}{100}

\DeclareMathSymbol{\lamed}{\mathord}{hebrewletters}{108}
\DeclareMathSymbol{\mem}{\mathord}{hebrewletters}{109}
\DeclareMathSymbol{\ayin}{\mathord}{hebrewletters}{96}
\DeclareMathSymbol{\tsadi}{\mathord}{hebrewletters}{118}
\DeclareMathSymbol{\qof}{\mathord}{hebrewletters}{114}
\DeclareMathSymbol{\shin}{\mathord}{hebrewletters}{152}

\begin{document}

\title{Frame Soft Shrinkage Operators are Proximity Operators}
\date{}
\author {Jakob Alexander Geppert$^{*}$ \quad and \quad  Gerlind Plonka\footnote{Institute for Numerical and Applied Mathematics, G\"ottingen University, Lotzestr.\ 16 -- 18,  37083 G\"ottingen, Germany. Email: \{j.geppert,plonka\}@math.uni-goettingen.de}}

\maketitle

\abstract{{\normalfont \bfseries Abstract.}
In this paper, we show that the commonly used frame soft shrinkage operator, that maps a given vector ${\mathbf x} \in {\mathbb R}^{N}$ onto the vector ${\mathbf T}^{\dagger} S_{\gamma} {\mathbf T}  {\mathbf x}$, is already a proximity operator, which can therefore be directly used in corresponding splitting algorithms. In our setting,  the frame transform matrix ${\mathbf T} \in {\mathbb R}^{L \times N}$ with $L \ge N$ has full rank $N$, ${\mathbf T}^{\dagger}$ denotes the Moore-Penrose inverse of ${\mathbf T}$, and $S_{\gamma}$ is the usual soft shrinkage operator  with threshold parameter $\gamma >0$.
Our result generalizes the known assertion that ${\mathbf T}^{*} S_{\gamma} {\mathbf T}$ is the proximity operator of $\| {\mathbf T} \cdot \|_{1}$
if ${\mathbf T}$ is an orthogonal (square) matrix.
It is well-known that for rectangular frame matrices ${\mathbf T}$ with $L > N$, the proximity operator of $\| {\mathbf T} \cdot \|_{1}$  does not have  a closed representation and needs to be computed iteratively.
We show that the frame soft shrinkage operator {${\mathbf T}^{\dagger} S_{\gamma} {\mathbf T}$} is a proximity operator as well,
thereby  motivating its application as a replacement of the exact proximity operator of $\| {\mathbf T} \cdot \|_{1}$.
We further give an explanation, why the usage of the frame soft shrinkage operator still provides good results in various applications.
 In particular, we provide some properties of the subdifferential of the convex functional $\Phi$ which leads to the proximity operator ${\mathbf T}^{\dagger} S_{\gamma} {\mathbf T}$ and show that ${\mathbf T}^{\dagger} S_{\gamma} {\mathbf T}$ approximates $\textrm{prox}_{\|{\mathbf T} \cdot\|_{1}}$.

\medskip

\textbf{Key words}: proximity operator; frame soft shrinkage; maximally cyclically monotone subdifferential;  splitting algorithms for inverse problems.
}

\section{Introduction}
Wavelet shrinkage and frame shrinkage operators, as e.g.\ curvelet and shearlet shrinkage are common tools in image denoising and reconstruction.
The underlying idea is that images are often piecewise smooth and can therefore be presented sparsely in these frames, cf.\ \cite{Easley_2008,Kittipoom_2012,Ma_2010}.
For discrete images of size $N_{1} \times N_{2}$ with $N=N_{1} \, N_{2}$, we can assume that the suitable image transform, which forces sparsity in the transformed domain, can be represented by a linear transform matrix  ${\mathbf T} \in {\mathbb R}^{L \times N}$ with $L \ge N$ of full rank $N$  which is applied to the vectorized image ${\mathbf f} \in {\mathbb R}^{N}$.
The simplest idea for wavelet or frame denoising of a noisy  image ${\mathbf f}$ consists of the following steps. First, we apply the wavelet/frame transform to obtain ${\mathbf T} \,{\mathbf f}$. Second,  we apply a soft shrinkage operator componentwise to eliminate the small wavelet/frame coefficients and obtain $S_{\gamma} \,({\mathbf T}\, {\mathbf f})$. Third, we apply the ``inverse'' transform
to obtain the denoised image ${\mathbf x}:={\mathbf T}^{-1} S_{\gamma} \, ({\mathbf T} \, {\mathbf f})$.
When ${\mathbf T}$ is not longer surjective and thus not invertible, the inverse transform ${\mathbf T}^{-1}$ will be replaced by the Moore Penrose inverse ${\mathbf T}^{\dagger}$, see e.g. \cite{Elad_2006,ER06,SCD02} and references therein.
Our ultimative goal in this paper is to show that
the frame soft shrinkage operator ${\mathbf T}^{\dagger} \, S_{\gamma} ({\mathbf T} \cdot )$ is the proximity operator of a proper, lower semi-continuous, convex functional ${\mathbf \Phi}$ and therefore can be directly linked to the solution of a minimization problem.  We start with  some notations and motivating comments.

Using an optimization approach, image denoising can be performed by asking for the minimizer $\hat{\mathbf x}$ of the functional
\begin{equation}\label{deno} \tfrac{1}{2}\|{\mathbf f} - {\mathbf x}\|_{2}^{2} + \gamma \, \|{\mathbf T}{\mathbf x}\|_{1}, \end{equation}
where ${\mathbf y}$ denotes the given noisy (vectorized) image, and we are looking to find some  ${\mathbf x}$ that minimizes this functional. Here, the regularization term $\|{\mathbf T}{\mathbf x}\|_{1}$ is taken in order to force the solution ${\mathbf x}$ to have a sparse expansion in the transformed domain.
In this context, $\| \cdot\|_{2}$ and $\| \cdot \|_{1}$ denote the Euclidean and the $1$-norm of vectors, i.e., $\| {\mathbf x}\|_{2}^{2} := \sum_{j=1}^{N} x_{j}^{2}$ and  $\| {\mathbf x}\|_{1} := \sum_{j=1}^{N} |x_{j}|$, respectively, and $\gamma >0$ is a regularization parameter.
Observe that the regularization term $\gamma \, \|{\mathbf T}{\mathbf x}\|_{1}$ is not differentiable with respect to $\mathbf{x}$, but still convex.
To study such minimization problems, we denote  the set of proper, lower semicontinuous, convex functionals  $\Phi: {\mathbb R}^{N} \to {\mathbb R}$ by $\Gamma_{0}$ and  introduce the so-called \emph{proximal mapping} or the \emph{proximity operator} for $\Phi \in \Gamma_{0}$,
\begin{equation}\label{prox} \textrm{prox}_{\Phi} ({\mathbf y}) \coloneqq \argmin_{{\mathbf x} \in {\mathbb R}^{N}} \left\{\tfrac{1}{2} \|{\mathbf y}- {\mathbf x} \|_{2}^{2} + \Phi({\mathbf x}) \right\}.
\end{equation}
In our special case, we have $\Phi({\mathbf x}) = \gamma \, \|{\mathbf T}{\mathbf x}\|_{1}$ such that  $\hat{\mathbf x} = \textrm{prox}_{\gamma \, \|{\mathbf T}\cdot\|_{1}} ({\mathbf f})$ is the denoising solution of (\ref{deno}).
To solve the minimization problem in (\ref{prox}) we use the  subdifferential $\partial \Phi$, which is defined as the set-valued operator
\begin{equation}\label{part}
\partial \Phi
({\mathbf x}) \coloneqq \{{\mathbf y} \in {\mathbb R}^{N}: \langle {\mathbf y}, \tilde{\mathbf x} - {\mathbf x} \rangle \le \Phi(\tilde{\mathbf x}) - \Phi({\mathbf x}), \, \forall \, \tilde{\mathbf x} \in {\mathbb R}^{N} \},
\end{equation}
where $\langle \cdot, \, \cdot \rangle$ denotes  a fixed scalar product in ${\mathbb R}^{N}$.
The subdifferential directly generalizes the usual notion of the derivative, and the minimizer $\textrm{prox}_{\Phi} ({\mathbf y}) $ of the functional in (\ref{prox})  necessarily satisfies
\begin{equation}\label{star}  0 \in \partial \left( \tfrac{1}{2}\|\textrm{prox}_{\Phi} ({\mathbf y}) - {\mathbf y}\|_{2}^{2} + \Phi(\textrm{prox}_{\Phi} ({\mathbf y})) \right)
= (\textrm{prox}_{\Phi} ({\mathbf y}) - {\mathbf y}) + \partial \Phi (\textrm{prox}_{\Phi} ({\mathbf y})),
\end{equation}
i.e.,
\begin{equation}\label{reso}  {\prox}_{\Phi} ({\mathbf y})= ( {\mathbf I}_{N} +  \partial \Phi)^{-1}({\mathbf y}),
\end{equation}
where ${\mathbf I}_{N}$ denotes the identity matrix. In particular, ${\prox}_{\Phi}$ is well-defined and single-valued, \cite{Ekeland}.
For the special case $\Phi({\mathbf x}) = \gamma \|{\mathbf x}\|_{1}$ (i.e., ${\mathbf T}$ is the identity operator),
the corresponding {proximity operator} turns out to be the well-known soft shrinkage operator,
\begin{equation}\label{soft} \prox_{\gamma \| \cdot \|_{1}} ({\mathbf y}) = S_{\gamma}({\mathbf y}) \quad  \textrm{with} \quad \left[ S_{\gamma} ({\mathbf y}) \right]_{j} \coloneqq  \begin{cases} y_{j}-\gamma, & y_{j} \ge \gamma, \\
y_{j} + \gamma, & y_{j} \le - \gamma,\\
0, & |y_{j}| < \gamma. \end{cases}
\end{equation}
For $\Phi({\mathbf x}) = \gamma \, \|{\mathbf T}{\mathbf x} \|_{1}$ with ${\mathbf T}$ being an orthonormal transform matrix, it can be shown that the corresponding proximity operator is
\begin{equation}
 \prox_{\gamma \| {\mathbf T} \cdot \|_{1}} ({\mathbf y}) = {\mathbf T}^{*} \, S_{\gamma} ({\mathbf T} \, {\mathbf y}) = {\mathbf T}^{-1} \, S_{\gamma} ({\mathbf T} \, {\mathbf y}),
\end{equation}
see Proposition 23.29 in \cite{Bauschke_2011}.  In other words, the soft shrinkage procedure ${\mathbf T}^{-1} \, S_{\gamma} {\mathbf T}$ explained in the beginning turns out to be equivalent to the solution of the functional minimization problem for any orthogonal transform matrix ${\mathbf T}$.
This, however is no longer true if ${\mathbf T}$ is not orthogonal or if $\mathbf{T}$ is not even surjective, i.e., if ${\mathbf T} \in {\mathrm R}^{L \times N}$ for $L > N$, see e.g.\ \cite{Elad_2006}.
In this case, the proximity operator of $\gamma \| {\mathbf T} \cdot \|_{1}$ is not equal to ${\mathbf T}^{*} \, S_{\gamma} ({\mathbf T} \, {\mathbf x})$.
It can no longer be represented in a closed form, and one has to make use of an iteration procedure to compute it.

But we can ask the following question. Can the solution of the frame soft shrinkage procedure ${\mathbf x}:={\mathbf T}^{\dagger} S_{\gamma} \, ({\mathbf T} \, {\mathbf y})$ with a frame transform matrix ${\mathbf T} \in {\mathrm R}^{L \times N}$ for $L > N$ be understood as the minimizer  of a functional similar to that in (\ref{deno})?
In other words, is ${\mathbf T}^{\dagger} S_{\gamma} \, {\mathbf T}$ a proximity operator?
And if yes, how far is it away from the minimizer of $\tfrac{1}{2} \|{\mathbf y}- {\mathbf x} \|_{2}^{2} + \gamma \, \|{\mathbf T}{\mathbf x}\|_{1}$?

A related question has been posed also by Elad in \cite{Elad_2006} from a different viewpoint. His main argument to explain the good performance of methods that just employ the frame soft shrinkage operator is based on  the connection to the solution of basis pursuit denoising (BPDN) problems \cite{chen2001}.
He showed that the application of the frame soft shrinkage operator can be interpreted as the first iteration step of an iterative algorithm to solve the BPDN problem.
\medskip

There are several motivations to study the question, whether the frame soft shrinkage operator is a proximity operator.\\
First, from the viewpoint of convex analysis, it is interesting to give a complete answer to the question, whether the concatenation of a proximity operator and a linear operator ${\mathbf T}^{\dagger} \, \prox_{\Phi} {\mathbf T}$ is still a proximity operator.
The answer had been known before for surjective transforms ${\mathbf T}$ satisfying ${\mathbf T} {\mathbf T}^{T} = \alpha {\mathbf I}$, see e.g.\ \cite{Beck17}, Section 6.

\smallskip

Second, knowing that the frame soft shrinkage operator is a proximity operator, it can be used as an activation function in the construction of neural networks.
In \cite{CP20}, it has been recently shown that all activation functions appearing in neural  networks are indeed proximity operators, see also \cite{HHN20}, Table 3.
In \cite{HHN20}, particularly the frame soft shrinkage operator with so-called Parseval frame matrices ${\mathbf T} \in {\mathbb R}^{L \times N}$, $L>N$, satisfying ${\mathbf T}^{T} {\mathbf T} = {\mathbf I}_{N}$ has been successfully employed.
\smallskip

Third, we are able to give an answer to the question raised above, namely that the frame soft  shrinkage for image denoising can be indeed understood as a minimizer of a variational problem and we can show that the denoising result is close to the minimizer of (\ref{deno}) also in the case of frame transform matrices ${\mathbf T}$ which are not surjective.

Finally, we mention that this observation is also interesting for image reconstruction problems, where sparsity of the transformed image ${\mathbf T} {\mathbf y}$ is used as a prior.

In this regard, one often studies the minimization of a functional of the form
\begin{equation}
\label{var}
 \hat{\mathbf x} = \argmin_{{\mathbf x} \in {\mathbb R}^{N}}  F({\mathbf x}) = \argmin_{{\mathbf x} \in {\mathbb R}^{N}} \left(\frac{1}{2} \|{\mathbf K} {\mathbf x} - {\mathbf f} \|_{2}^{2} +  \Phi({\mathbf x})\right),
\end{equation}
where  ${\mathbf K}\colon {\mathbb R}^{N} \to {\mathbb R}^{M}$ is a bounded linear operator,  e.g.\ a blurring operator, and ${\mathbf f} \in {\mathbb R}^{M}$ represents the measured (noisy) data,
see e.g.\ \cite{PM11,Vande}. In the compressed sensing approach, we usually have $M < N$, and a meaningful reconstruction strongly relies on the prior information of the image.
Often the regularization functional $\Phi({\mathbf x})$ ist taken in the form  $\Phi({\mathbf x}) = \gamma \, \| {\mathbf T} {\mathbf x}\|_{1}$ with some suitable matrix ${\mathbf T}$ that can be a frame transform matrix or a matrix built by a concatenation of two different discrete bases or frames, see e.g.\ \cite{PM11}. Another example is the anisotropic TV model
for a matrix ${\mathbf X}= (x_{jk})_{j,k=1}^{N_{1}, N_{2}}$,
$$ TV({\mathbf X}) \coloneqq \sum_{j=1}^{N_{1}-1}\sum_{k=1}^{N_{2}-1} (|x_{j+1,k} - x_{j,k}| + |x_{j,k+1}-x_{j,k}|)$$
that can be rewritten as $\|{\mathbf T} {\mathbf x}\|_{1}$ for the vectorized image ${\mathbf x} = \textrm{vec} \, ({\mathbf X}) \in {\mathbb R}^{N_{1}N_{2}}$ with some ${\mathbf T} \in {\mathbb R}^{2N_{1}N_{2}-N_{1}-N_{2}}$.
However, the lacking knowledge about a closed representation of the proximity operator of $\| {\mathbf T} {\mathbf x}\|_{1}$ strongly complicates the minimization problem (\ref{var}). Therefore one often considers a constrained optimization problem instead,
$$ \min_{{\mathbf x} \in {\mathbb R}^{N}, {\mathbf z} \in {\mathbb R}^{L}} \left(\tfrac{1}{2} \|{\mathbf K} {\mathbf x} - {\mathbf f} \|_{2}^{2} +  \gamma \|{\mathbf z}\|_{1}\right) \qquad \textrm{subject~to} \quad  {\mathbf T} {\mathbf x} = {\mathbf z}, $$
and employs and an augmented Lagrangian approach, since the proximity operator of $\|{\mathbf z}\|_{1}$, is known, see (\ref{soft}).
Nowadays this problem is usually solved via a primal-dual algorithm \cite{CP11}, a (preconditioned) ADMM \cite{Esser} or a split Bregman algorithm \cite{Osher, PM11}. These approaches are closely related and are equivalent under certain conditions, \cite{Setzer}.

If the proximity operator of $\Phi({\mathbf x})$ was known, the minimization problem (\ref{var}) could be simply solved by a forward-backward splitting method as follows:
The solution $\hat{\mathbf x}$ of (\ref{var})  necessarily satisfies
$$  0 \in \partial \left( \frac{1}{2} \|{\mathbf K} {\mathbf x} - {\mathbf f}\|_{2}^{2} +  \Phi({\mathbf x}) \right) = {\mathbf K}^{*} ( {\mathbf K} {\mathbf x} - {\mathbf f}) + \, \partial \Phi ({\mathbf x}).$$

\noindent
Multiplication with a constant $\lambda >0$ and addition of $\hat{\mathbf x}$  yields the equivalent statements
\begin{align*}
\hat{\mathbf x} -  \lambda \, {\mathbf K}^{*} ({\mathbf K} \hat{\mathbf x} - {\mathbf f})  &\in  \hat{\mathbf x} +  \lambda \, \partial \Phi (\hat {\mathbf x}) \\
({\mathbf I}_{N} -  \lambda \, {\mathbf K}^{*} {\mathbf K}) \hat{\mathbf x} +  \lambda \, {\mathbf K}^{*} {\mathbf f} & \in ({\mathbf I}_{N} +  \lambda \, \partial \Phi)^{-1}(\hat {\mathbf x}),
\end{align*}
with the $N \times N$ identity matrix ${\mathbf I}_{N}$. Thus, using (\ref{reso}), it follows that
\begin{equation}\label{ite1}
\hat{\mathbf x} = ( {\mathbf I}_{N} + \lambda \, \partial \Phi)^{-1} \left[ ({\mathbf I}_{N} -  \lambda \, {\mathbf K}^{*} {\mathbf K}) \hat{\mathbf x} +  \lambda \, {\mathbf K}^{*} {\mathbf f} \right] = \textrm{prox}_{\lambda \Phi} \left[ ({\mathbf I}_{N} - \lambda\,  {\mathbf K}^{*} {\mathbf K}) \hat{\mathbf x} +  \lambda \, {\mathbf K}^{*} {\mathbf f} \right].
\end{equation}

This yields the well-known forward-backward splitting iteration, see e.g.\ \cite{Lions79, Bauschke_2011, Setzer}, which converges for $\lambda \in (0, 2/\|{\mathbf K}\|_{2}^{2})$:

\begin{algorithm}
$(\textbf{Forward-backward splitting})$ \\
For an arbitrary starting vector ${\mathbf x}^{(0)} \in {\mathbb R}^{N}$, iterate
\begin{enumerate}
\item $ {\mathbf y}^{(j)} \coloneqq ({\mathbf I}_{N} -  \lambda \, {\mathbf K}^{*} {\mathbf K}) \, {\mathbf x}^{(j)} +  \lambda \, {\mathbf K}^{*} {\mathbf f}$,
\item $ {\mathbf x}^{(j+1)} \coloneqq {\prox}_{\lambda \Phi}({\mathbf y}^{(j)}). $
\end{enumerate}
\end{algorithm}
\medskip

Our result now allows us to take just  ${\prox}_{\lambda \Phi}({\mathbf y}^{(j)})= \lambda \, {\mathbf T} S_{\gamma} ({\mathbf T} {\mathbf y}^{(j)})$ in this iteration.
This is \emph{not} equivalent with taking the proximity operator of $\|{\mathbf T} {\mathbf y}^{(j)}\|_{1}$, but our results show that it is close!
We need to keep in mind that we want to enforce sparsity of ${\mathbf T}  {\mathbf x}$ using the functional $\Phi({\mathbf x} ) = \gamma \, \| {\mathbf T} {\mathbf x} \|_{1}$.
However, the used $\ell_{1}$-norm only acts as a proxy here -- a compromise to the convexity for $\Phi$ -- which is needed in order to ensure convergence of the iteration algorithms.
Instead, we would rather like to have actual sparsity of ${\mathbf T} {\mathbf x}$, i.e., a  small number of non-zero components in  ${\mathbf T} {\mathbf x}$.
Thus, one might wonder whether ${\mathbf T}^{\dagger} \, S_{\gamma} ({\mathbf T} \cdot )$ is doing the job as well as the exact proximity operator of $\| {\mathbf T} \, {\mathbf x} \|_{1}$.

Our results e.g.\ in \cite{Loock_2014,Loock_2016} imply similarly as earlier results in \cite{Elad_2006}, that this frame soft shrinkage works well in practice.
Knowing now that it is indeed a proximity operator, we can conclude that the convergence results that have been shown for the known iterative algorithms
for functional minimization can be applied if we use  frame soft shrinkage.

Note that the solution of (\ref{var}) for image reconstruction via forward-backward splitting is different from the so-called synthesis approach, where one simply considers ${\mathbf z}:={\mathbf T} {\mathbf x}$ and solves
$\textrm{argmin}_{{\mathbf z}} \,  ( \|{\mathbf K'} {\mathbf z} - {\mathbf T}{\mathbf y}\|_{2}^{2} + \gamma \, \| {\mathbf z}\|_{1})$ in the transformed domain with ${\mathbf K}'= {\mathbf K} {\mathbf T}^{\dagger}$, see \cite{Elad_2006,EMR07,BL08}.
\smallskip

 Based on the results in this paper for the frame soft shrinkage operator, we have been able to show the following more general result in the follow-up paper \cite{HHN20}:
Let ${\mathcal H}$ and ${\mathcal K}$ be two Hilbert spaces with scalar products $\langle \cdot, \cdot \rangle_{{\mathcal H}}$ and
$\langle \cdot, \cdot \rangle_{{\mathcal K}}$, $b \in {\mathcal K}$, ${T}: {\mathcal H} \to {\mathcal  K}$ a bounded linear operator with closed range,
and $\textrm{prox} : {\mathcal K} \to {\mathcal K}$  a proximity operator  on ${\mathcal K}$. Then ${T}^{\dagger} \, \textrm{prox} ({T} \cdot + b): {\mathcal H}_{{T}}  \to {\mathcal H}_{{T}}$ is a proximity operator,
where ${\mathcal H}_{{T}}$ denotes the Hilbert space ${\mathcal H}$ equipped
with the modified scalar product $\langle x, y \rangle_{{\mathcal H}_{T}} = \langle {T} x, {T} y \rangle_{{\mathcal K}} + \langle {\mathcal P}_{T}x, {\mathcal P}_{T}y \rangle_{\mathcal H}$,
and where ${\mathcal P}_{T}$ is the projection onto the kernel of ${T}$.

The results of the present paper  are however not contained in \cite{HHN20}.  Section 2 focusses on  the properties of the  subdifferential $H$ which is related to  ${\mathbf T}^{\dagger} S_{\gamma} {\mathbf T}$.  Further, our proof that ${\mathbf T}^{\dagger} S_{\gamma} {\mathbf T}$ is indeed a proximity operator essentially differs from  the proof given in \cite{HHN20}, where special properties of the Moreau envelope of functions in $\Gamma_{0}$ are used.

\medskip

This paper is structured as follows.
In Section 2, we introduce a set-valued  mapping $H$ which closely related to  the frame soft shrinkage operator ${\mathbf T}^{\dagger} \, S_{\gamma} ({\mathbf T} \cdot )$. We will show that  $H= \partial \Phi$ for some functional $\Phi \in \Gamma_{0}$ such that $\textrm{prox}_{\Phi}= {\mathbf T}^{\dagger}
 S_{\gamma }({\mathbf T} \cdot)$. We start with showing  that $H$ is well-defined and we derive some structural properties of $H$.
In particular, we prove that $H$ possesses similar properties as the  subdifferential of $ \|{\mathbf T} \cdot \|_{1}$.
In Section 3, we show that the set-valued mapping $H$ is maximally cyclically monotone and therefore indeed the subdifferential of a proper, lower semi-continuous and convex functional $\Phi$.
 The key idea to achieve this result is to apply a new scalar product in ${\mathbb R}^{N}$ which is aligned  with the linear operator ${\mathbf T}$.
We will be able to conclude that the frame soft shrinkage operator is indeed the proximity operator of a functional $\Phi$ with $H = \partial \Phi$ and particularly is non-expansive with respect to the aligned scalar product.

\section{A closer look at the frame soft shrinkage operator}

Our goal is to show that  for any frame matrix ${\mathbf T} \in {\mathbb R}^{L \times N}$  with $L \ge N$ and full rank $N$  the operator
$$ {\mathbf T}^{\dagger} \, S_{\gamma} {\mathbf T},$$
with $S_{\gamma}$ the soft threshold operator in (\ref{soft}), is the proximity operator of a convex, proper, lower semi-continuous functional, i.e., of some $\Phi \in \Gamma_{0}$.
 The full rank assumption on ${\mathbf T}$ implies that ${\mathbf T}^{*} {\mathbf T}$ is invertible and therefore ${\mathbf T}^{\dagger}= ({\mathbf T}^{*} {\mathbf T})^{-1} {\mathbf T}^{*}$, where ${\mathbf T}^{*}$ denotes the transpose of ${\mathbf T}$. Further, we have ${\mathbf T}^{\dagger} {\mathbf T} = {\mathbf I}_{N}$.
Note that the assumption that ${\mathbf T}$ has full rank $N$ is meaningful in this regard. If ${\mathbf T}$ is not injective, then any ${\mathbf x}$ in the kernel of ${\mathbf T}$ would be mapped to zero by ${\mathbf T}^{\dagger} \, S_{\gamma} {\mathbf T}$.

{
Recall that for a function $\Phi \in \Gamma_{0}$ it follows from (\ref{prox}) and (\ref{star}) that $0 \in \textrm{prox} ({\mathbf z}) - {\mathbf z} + \partial \Phi(\textrm{prox} ({\mathbf z}))$, i.e.,
$$  {\mathbf z} - {\mathbf x} \in \partial \Phi({\mathbf x})  \quad \Longleftrightarrow  \quad     {\mathbf x} = \textrm{prox} ({\mathbf z}),
 $$
or equivalently, with ${\mathbf z}= {\mathbf x} + {\mathbf y}$,
\begin{equation}\label{prox1} {\mathbf y} \in \partial \Phi({\mathbf x})  \quad \Longleftrightarrow  \quad     {\mathbf x} = \textrm{prox} ({\mathbf x} + {\mathbf y}),
\end{equation}
see  Proposition 16.34 in \cite{Bauschke_2011}.
Therefore,} we define in  a first step the set-valued mapping ${H}= H_{\gamma}\colon\mathbb{R}^N\rightrightarrows\mathbb{R}^N$ by
\begin{equation}\label{H}
{\mathbf y} \in H({\mathbf x}) \quad : \Longleftrightarrow \quad {\mathbf x} = {\mathbf T}^{\dagger} S_{\gamma} {\mathbf T} \, ({\mathbf x} + {\mathbf y}).
\end{equation}
Then, we have to show that $H$ is the subdifferential of a functional $\Phi \in \Gamma_{0}$.

\noindent
In this section, we will show that $H$ in (\ref{H}) is well-defined, and we will study some properties of $H$. In Section \ref{sec:4}, we will finally show that
indeed $H = \partial \Phi$ for some $\Phi \in \Gamma_{0}$.
In order to get a first idea of what happens here, we start off with a toy example.

\begin{lemma}\label{ex1} For ${\mathbf T}= \begin{pmatrix} 1 \\ c \end{pmatrix}$ with $c \ge 1$ and $\gamma >0$ we find for $H$ in $(\ref{H})$ for $x \ge 0$
$$ H({x}) = \begin{cases}
\gamma [- \frac{1}{c}, \, \frac{1}{c}] & x =0, \\[1mm]
\frac{\gamma}{c} + \frac{x}{c^{2}} & x \in (0, \frac{\gamma(c-1)c}{c^{2}+1} ] ,\\[1mm]
\gamma \left( \frac{1+c}{1+c^{2}} \right) & x > \frac{\gamma (c-1)c}{c^{2}+1}.
\end{cases}
$$
For $x<0$ we have $H(x) = - H(-x)$.
Then $H$ is the subdifferential of the even function
$$\Phi(x) =  \begin{cases}
 \frac{\gamma x}{c} +  \frac{x^{2}}{2c^{2}}  & x \in [0, \frac{\gamma(c-1)c}{c^{2}+1} ], \\
\gamma \left( \frac{1+c}{1+c^{2}} \right) x - \frac{\gamma^2 (c-1)^2}{2 (c^2+1)^2} & x > \frac{\gamma(c-1)c}{c^{2}+1}, \\
\Phi(-x) & x < 0. \end{cases}
$$
\end{lemma}

\begin{proof}
For $x=0$, it follows from (\ref{H}) with  ${\mathbf  T}^{\dagger}= \frac{1}{1+c^{2}} (1, c)$ that
$$
y \in H(0) \quad  \Longleftrightarrow \quad 0 = \frac{1}{1+c^{2}} (1, c) \, S_{\gamma} \,  \begin{pmatrix} 1 \\ c \end{pmatrix} y.
$$
This is only true iff $S_{\gamma} \, (c\, y) =0$, i.e., $ y \in [-\frac{\gamma}{c}, \frac{\gamma}{c}]$.
For $x >0$ we find
$$ x = \frac{1}{1+c^{2}} (1, c) \, S_{\gamma} \, \begin{pmatrix} 1 \\ c \end{pmatrix} (y+x). $$
Since $x>0$, it necessarily follows that $c(x+y) > \gamma$ as well as $x+y >0$.
We consider two cases.

\noindent
1. If $x+y \le \gamma$ and $c(x+y) > \gamma$, then
$$ x = {\mathbf T}^{\dagger} S_{\gamma} {\mathbf T} (x+y) = \frac{1}{1+c^{2}} (1, c) \,  \begin{pmatrix} 0 \\ c(x+y) - \gamma \end{pmatrix}
= \frac{c^{2}(x+y) - c \gamma}{1 + c^{2}}  $$
implies
$$ y = \frac{x}{c^{2}} + \frac{\gamma}{c}. $$
Further, the condition
$ x+y  = x + (\frac{x}{c^{2}} + \frac{\gamma}{c}) \le \gamma$
 yields $ x \le \frac{c(c-1) \gamma}{1+c^{2}}$.

\noindent
2. Let now $x+y > \gamma$ and $c(x+y)  >\gamma$, then
$$ x = {\mathbf T}^{\dagger} S_{\gamma} {\mathbf T} (x+y) = \frac{1}{1+c^{2}} (1,c) \,  \begin{pmatrix} x+y - \gamma \\ c(x+y) - \gamma \end{pmatrix}  = x+y - \frac{\gamma(c+1)}{c^{2} + 1}. $$
Thus, we find $y = \frac{\gamma(c+1)}{c^{2} + 1}$, and $x+y > \gamma$ is true for $x >  \frac{\gamma c (c-1)}{c^{2} + 1}$.
Similar considerations for $x<0$ yield $H(-x) = -H(x)$.
\smallskip

\noindent
Integration  gives
$ \Phi(x) $ as asserted, where the constant is taken such that $\Phi(x)$ is continuous at $x=\frac{\gamma(c-1)c}{c^{2}+1}$.
\end{proof}

\begin{example}\label{ex}
If we employ Lemma $\ref{ex1}$ for $c=2$ and $\gamma = \frac{5}{3}$, we find
$$ H(x) = \begin{cases}
[-\frac{5}{6}, \, \frac{5}{6}] & x=0, \\
\frac{5}{6} + \frac{x}{4} & x \in (0, \frac{2}{3}], \\
1 & x >\frac{2}{3},\\
-H(-x) & x <0, \end{cases} \qquad\text{and}\qquad
\Phi(x) = \begin{cases}
\frac{5x}{6} + \frac{x^{2}}{8} & x \in [0, \, \frac{2}{3}] \\
x- \frac{1}{18} & x > \frac{2}{3} \\
\Phi(-x) & x < 0. \end{cases}
$$
Thus $\Phi(x)$ approximates $|x|$, see Figure $\ref{figH}$.
\end{example}

\begin{figure}[ht]
 \centering
      \includegraphics[width=.27\textwidth]{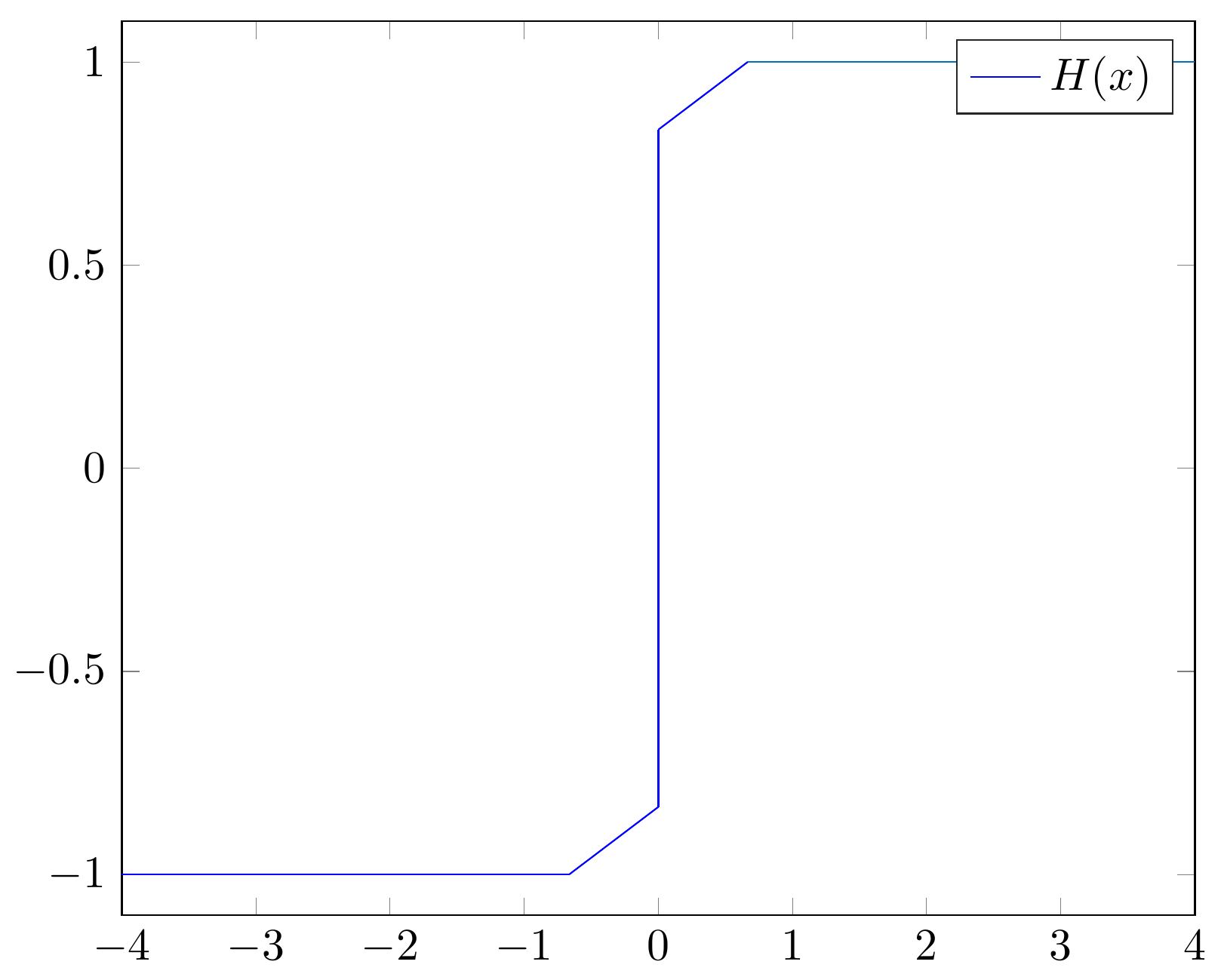}\quad
  \includegraphics[width=.26\textwidth]{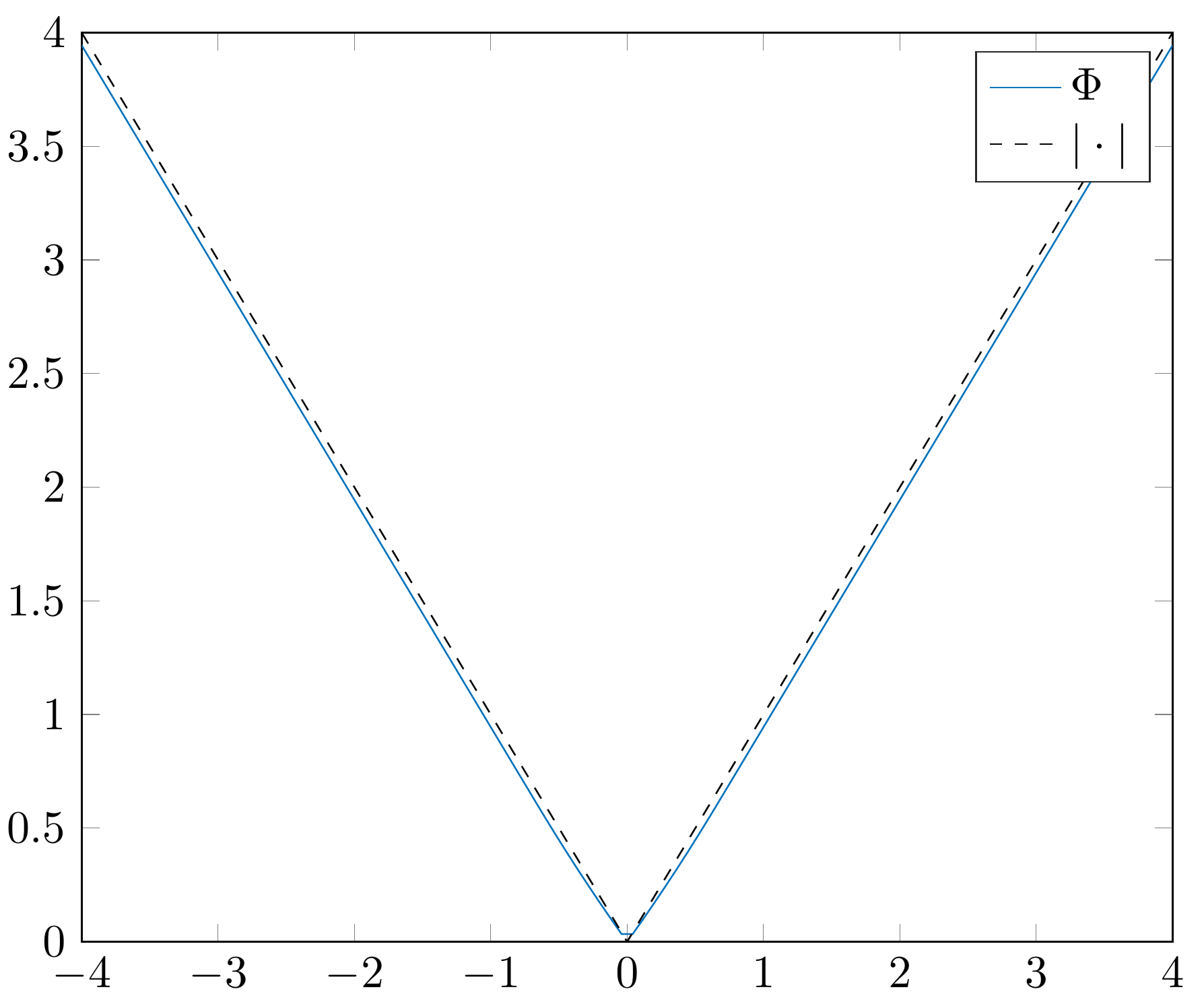}   \quad
  \includegraphics[width=.26\textwidth]{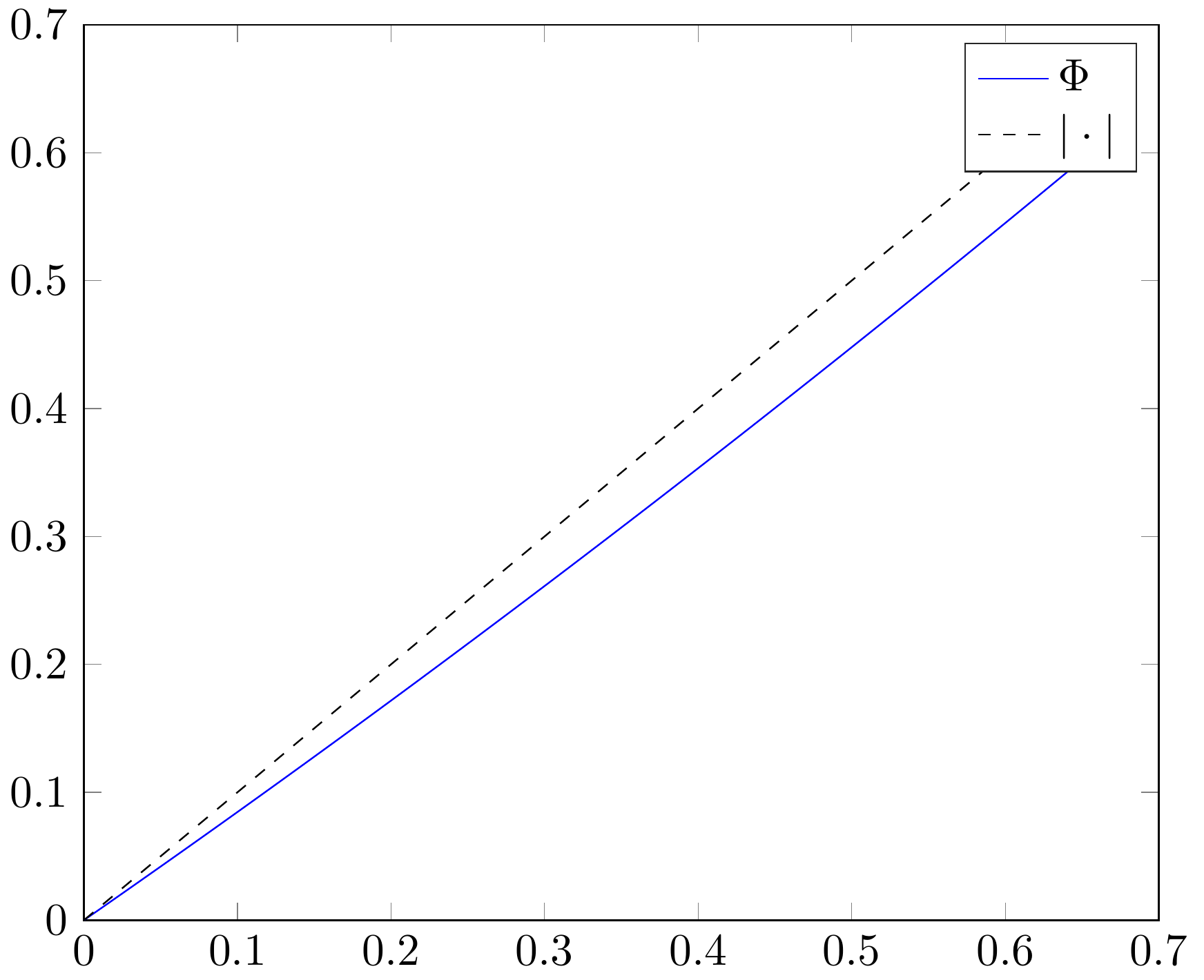}
     \\[1ex]
     \caption{Visualization of  $H(x)$ and $\Phi(x)$ in Example \ref{ex}.}
\label{figH}
\end{figure}

Before starting to inspect the function $H$ in (\ref{H}) more closely, we want to show that ${\mathbf T}^{\dagger} S_{\gamma} {\mathbf T}$ is a firmly
nonexpansive mapping if we take a suitable scalar product in ${\mathbb R}^{N}$.
 Since we have assumed that ${\mathbf T} \in {\mathbb R}^{L \times N}$ with $L \ge N$ has full rank $N$, it follows that ${\mathbf T}^{*}{\mathbf T}$ is positive definite, and we can define the new scalar product and the corresponding norm in ${\mathbb R}^{N}$  by
\begin{equation}\label{normneu}
\langle {\mathbf x}, \, {\mathbf y} \rangle_{{\mathbf T}} := {\mathbf x}^{T}\,  {\mathbf T}^{*}{\mathbf T}\,  {\mathbf y}, \qquad \|{\mathbf x}\|_{\mathbf T}^{2} := \| {\mathbf T}{\mathbf x}\|_{2}^{2}, \end{equation}
where here $\| {\mathbf T}{\mathbf x} \|_{2}$ denotes the Euclidian norm of ${\mathbf T}{\mathbf x}$ in ${\mathbb R}^{L}$.
Further, for any vector ${\mathbf z} \in {\mathbb R}^{L}$ we use the notation
$$ \textrm{sign}\, {\mathbf z} \coloneqq ( \textrm{sign} \, z_{j})_{j=1}^{L} \quad \textrm{with} \quad \textrm{sign} \, z_{j} \coloneqq \left\{\begin{array}{rl} 1 & z_{j} > 0, \\ -1 & z_{j} < 0, \\ \left[-1,1 \right] & z_{j} =0. \end{array} \right.
$$
Then we have

\begin{proposition}\label{prop}
Let ${\mathbf T} \in {\mathbb R}^{L \times N}$ with $L \ge N$ and full rank $N$.  Then  the operators
$ {\mathbf T}^{\dagger} S_{\gamma} {\mathbf T} $ and ${\mathbf I}- {\mathbf T}^{\dagger} S_{\gamma} {\mathbf T}$ are firmly non-expansive, i.e., for all ${\mathbf x}, {\mathbf y} \in {\mathbb R}^{N}$ we have
$$
\|{\mathbf T}^{\dagger} S_{\gamma} {\mathbf T} {\mathbf x} - {\mathbf T}^{\dagger} S_{\gamma} {\mathbf T} {\mathbf y}\|^{2}_{{\mathbf T}} + \|({\mathbf I} - {\mathbf T}^{\dagger} S_{\gamma} {\mathbf T}){\mathbf x}- ({\mathbf I} - {\mathbf T}^{\dagger} S_ {\gamma} {\mathbf T}){\mathbf y} \|^{2}_{{\mathbf T}}  \le \|{\mathbf x} - {\mathbf y} \|^{2}_{{\mathbf T}}.
$$
\end{proposition}

\begin{proof} Since ${\mathbf T} {\mathbf T}^{\dagger}  = {\mathbf T} ({\mathbf T}^{*} {\mathbf T})^{-1} {\mathbf T}^{*} \in {\mathbb R}^{L \times L}$ is a projector, it follows that
$$ \|{\mathbf T}^{\dagger} S_{\gamma} {\mathbf T} {\mathbf x} - {\mathbf T}^{\dagger} S_{\gamma} {\mathbf T} {\mathbf y}\|^{2}_{{\mathbf T}} = \|({\mathbf T} {\mathbf T}^{\dagger}  (S_{\gamma} {\mathbf T} {\mathbf x} - S_{\gamma} {\mathbf T} {\mathbf y}) \|_{2}^{2} \le \|S_{\gamma} {\mathbf T} {\mathbf x} - S_{\gamma} {\mathbf T} {\mathbf y} \|_{2}^{2}$$
as well as
\begin{eqnarray*}
& &  \|({\mathbf I} - {\mathbf T}^{\dagger} S_{\gamma} {\mathbf T}){\mathbf x}- ({\mathbf I} - {\mathbf T}^{\dagger} S_{\gamma} {\mathbf T}){\mathbf y} \|^{2}_{{\mathbf T}} =
\|{\mathbf T} ({\mathbf x} - {\mathbf y}) - {\mathbf T} {\mathbf T}^{\dagger}( S_{\gamma} {\mathbf T} {\mathbf x} - S_{\gamma} {\mathbf T} {\mathbf y}) \|_{2}^{2} \\
&\le& \|{\mathbf x} - {\mathbf y} \|^{2}_{{\mathbf T}} + \|S_{\gamma} {\mathbf T} {\mathbf x} - S_{\gamma} {\mathbf T} {\mathbf y} \|_{2}^{2} - 2 \langle {\mathbf T} ({\mathbf x} - {\mathbf y}) , {\mathbf T} {\mathbf T}^{\dagger}(S_{\gamma} {\mathbf T} {\mathbf x} - S_{\gamma} {\mathbf T} {\mathbf y}) \rangle_{2}.
\end{eqnarray*}
To prove the assertion of the proposition we use that  ${\mathbf T}^{*} {\mathbf T} {\mathbf T}^{\dagger}={\mathbf T}^{*}$. Then it suffices to show  that
$$ \|S_{\gamma} {\mathbf T} {\mathbf x} - S_{\gamma} {\mathbf T} {\mathbf y} \|_{2}^{2} -  \langle {\mathbf T} ({\mathbf x} - {\mathbf y}) , S_{\gamma} {\mathbf T} {\mathbf x} - S_{\gamma} {\mathbf T} {\mathbf y} \rangle_{2} \le 0. $$
But this assertion is obviously true since for each component $(S_{\gamma} {\mathbf T} {\mathbf x} - S_{\gamma} {\mathbf T} {\mathbf y})_{k}$,  $k=1, \ldots , L$, we have by (\ref{soft}) either
$(S_{\gamma} {\mathbf T} {\mathbf x} - S_{\gamma} {\mathbf T} {\mathbf y})_{k} = 0$
or $\textrm{sign}( S_{\gamma} {\mathbf T} {\mathbf x} - S_{\gamma} {\mathbf T}
{\mathbf y})_{k}) = \textrm{sign}( {\mathbf T} ({\mathbf x} - {\mathbf y}))_{k}$
and $| (S_{\gamma} {\mathbf T} {\mathbf x} - S_{\gamma} {\mathbf T} {\mathbf y})_{k}| \le |( {\mathbf T} ({\mathbf x} - {\mathbf y}))_{k}|$.
\end{proof}
\medskip

\begin{remark}
Proposition $\ref{prop}$ particularly implies that $ {\mathbf T}^{\dagger} S_{\gamma} {\mathbf T} $ and ${\mathbf I}- {\mathbf T}^{\dagger} S_{\gamma} {\mathbf T}$ are non-expansive with respect to $\| \cdot \|_{{\mathbf T}}$, see \cite{Bauschke_2011}.
\end{remark}

Now we inspect the function $H$ in (\ref{H}) and show first that it is well defined, i.e., for any fixed $\gamma >0$ and each ${\mathbf x} \in {\mathbb R}^{N}$ the set $H({\mathbf x})$ is not empty.
By (\ref{H}), we have ${\mathbf y} \in H({\mathbf x})$ if ${\mathbf x} = {\mathbf T}^{\dagger} S_{\gamma} {\mathbf T} ({\mathbf x} + {\mathbf y})$. With the substitution ${\mathbf t} = {\mathbf x} + {\mathbf y}$, we get the equivalent fixed point representation
$$
{\mathbf t} - {\mathbf x} \in H({\mathbf x}) \qquad  \Longleftrightarrow \qquad {\mathbf t} = {\mathbf x} + ({\mathbf I}_{N} -  {\mathbf T}^{\dagger} S_{\gamma} {\mathbf T}) \, {\mathbf t}. $$
Thus, if the function $f_{{\mathbf x}, {\mathbf T}}\colon {\mathbb R}^{N} \longrightarrow {\mathbb R}^{N}$ with
\begin{equation}\label{fxT}
 f_{{\mathbf x}, {\mathbf T}} ({\mathbf t}) \coloneqq {\mathbf x} + ({\mathbf I}_{N} - {\mathbf T}^{\dagger} S_{\gamma} {\mathbf T}) \, {\mathbf t}
\end{equation}
possesses a fixed point ${\mathbf t}$, then ${\mathbf y} = {\mathbf t}-{\mathbf x}$ is an element of $H({\mathbf x})$.

\begin{theorem}\label{Hwell}
Let ${\mathbf T} \in {\mathbb R}^{L \times N}$ with $L \ge N$ have full rank $N$ and let $\gamma >0$.
Then, for each ${\mathbf x} \in {\mathbb R}^{N}$, we have  $H({\mathbf x}) \neq \emptyset$.
Further,  the image of $H$ is bounded, i.e., for each ${\mathbf x} \in {\mathbb R}^{N}$ we have $H({\mathbf x}) \subset \{{\mathbf y} \in {\mathbb R}^{N}: \, \|{\mathbf y} \|_{2} \le \gamma \,  \sqrt{L}\, \|{\mathbf T} \|_{2} \}$, where $\|{\mathbf T}\|_{2}$ denotes the spectral norm of ${\mathbf T}$.
\end{theorem}

\begin{proof}
To prove that $H({\mathbf x}) \neq \emptyset$, we show that  for each ${\mathbf x} \in {\mathbb R}^{N}$  the function $f_{{\mathbf x}, {\mathbf T}}$ possesses at least one fixed point.
We define the closed ball
$$ B({\mathbf x}, \gamma \sqrt{L}) \coloneqq \{{\mathbf t} \in {\mathbb R}^{N}:  \, \| {\mathbf x} - {\mathbf t}\|_{\mathbf T} \le \gamma \, \sqrt{L} \}. $$
Then, using ${\mathbf T} {\mathbf T}^{\dagger} {\mathbf T} = {\mathbf T}$ and $\| {\mathbf T} {\mathbf T}^{\dagger}\|_{2} = 1$, we have
\begin{align*}
\sup_{{\mathbf t} \in {\mathbb R}^{N}} \|{\mathbf x} - f_{{\mathbf x},{\mathbf T}}({\mathbf t})\|_{{\mathbf T}}
 &= \sup_{{\mathbf t} \in {\mathbb R}^{N}} \|({\mathbf I}_{N} - {\mathbf T}^{\dagger} S_{\gamma} {\mathbf T}) {\mathbf t}\|_{{\mathbf T}}  =
 \sup_{{\mathbf t} \in {\mathbb R}^{N}} \| ({\mathbf T}  - {\mathbf T} {\mathbf T}^{\dagger} S_{\gamma} {\mathbf T}) {\mathbf t}\|_{2} \\
&\le  \sup_{{\mathbf s} \in {\mathbb R}^{L}} \| {\mathbf s} - S_{\gamma} {\mathbf s}\|_{2} \le \gamma  \|{\mathbf 1}\|_{2} = \gamma \sqrt{L},
 \end{align*}
 where ${\mathbf 1} \in {\mathbb R}^{L}$ denotes the vector with ones as components, and where we have used the definition (\ref{soft}) of $S_{\gamma}$. Therefore, $f_{{\mathbf x}, {\mathbf T}}({\mathbf t}) \in B({\mathbf x}, \gamma \sqrt{L})$ for any ${\mathbf t} \in {\mathbb R}^{N}$.
Further, $f_{{\mathbf x}, {\mathbf T}}$ is nonexpansive with regard to $\| \cdot \|_{\mathbf T}$ by Proposition \ref{prop}. Since $S_{\gamma}$ and thus also $f_{{\mathbf x}, {\mathbf T}}$ is continuous, it follows by Browders  fixed point theorem that $f_{{\mathbf x}, {\mathbf T}}$ possesses a fixed point in $B({\mathbf x}, \gamma \sqrt{L})$, see \cite{Bauschke_2011}, Theorem 4.19.
Hence, for each ${\mathbf x} \in {\mathbb R}^{N}$ we have ${\mathbf y} = {\mathbf t} - {\mathbf x} \in H({\mathbf x})$ for all fixed points ${\mathbf t}$ of $f_{{\mathbf x}, {\mathbf T}}$, i.e. $H({\mathbf x}) \neq \emptyset$.
In particular  the image of $H$ is bounded, and we have
$$ \| H({\mathbf x}) \|_{\mathbf T} = \| f_{{\mathbf x}, {\mathbf T}}({\mathbf t}) - {\mathbf x}\|_{\mathbf T}  \le \| {\mathbf T}\|_{2} \| f_{{\mathbf x}, {\mathbf T}}({\mathbf t}) - {\mathbf x}\|_{2}
\le \gamma  \sqrt{L} \, \| {\mathbf T}\|_{2}$$
as seen above.
\end{proof}
\medskip

Thus, we conclude that the mapping ${H}\colon\mathbb{R}^N\rightrightarrows\mathbb{R}^N$ in (\ref{H}) is well defined.
In the remaining part of this section we prove some further properties of $H({\mathbf x})$ which show that $H$ behaves similarly as the subdifferential of $\|{\mathbf T} \cdot\|_{1}$ for an orthogonal transform matrix ${\mathbf T}$. 
Next we show that $H({\mathbf 0})$ is indeed not single-valued.

\begin{theorem}\label{H0}
  Let ${\mathbf T} \in\mathbb{R}^{L\times N}$  with $L \ge N$ with full rank $N$.
  Further let $\gamma > 0$ and  $H$  as in $(\ref{H})$.
  Then  ${\mathbf y} \in H({\mathbf 0})$ if and only if $\|{\mathbf T} \, {\mathbf y}\|_\infty \leq \gamma$,
  where
  $ \|{\mathbf T} \, {\mathbf y}\|_{\infty} \coloneqq \max\limits_{j \in \{1, \ldots, L\}} |[{\mathbf T}{\mathbf y}]_{j}|.$
\end{theorem}

\begin{proof}
  We recall that ${\mathbf T}^\dagger {\mathbf T} = ({\mathbf T}^* {\mathbf T})^{-1}{\mathbf T}^* {\mathbf T}  = {\mathbf I}_{N}$ as ${\mathbf T}$ has full rank $N$.\\
First, let $\|{\mathbf T} \, {\mathbf y}\|_\infty \leq \gamma$.
  Then the definition of $S_{\gamma}$ in (\ref{soft}) implies  $S_\gamma {\mathbf T}\, {\mathbf y} = {\mathbf 0}$ and hence also ${\mathbf T}^\dagger S_\gamma {\mathbf T} \, ({\mathbf y}+ {\mathbf 0}) = {\mathbf 0}$, that is, $ {\mathbf y} \in H({\mathbf 0})$.\\[1ex]
Second, let ${\mathbf y} \in H({\mathbf 0})$, i.e.,
  \begin{equation}\label{tst0}
   {\mathbf T}^\dagger S_\gamma {\mathbf T} \, {\mathbf y} = {\mathbf 0}.
  \end{equation}
  We show that then $\|{\mathbf T} \, {\mathbf y}\|_\infty \leq \gamma$.
  We consider the components $[{\mathbf T} \, {\mathbf y}]_{j}$, $j=1, \ldots , L$, and define three index sets $I_1, I_2, I_3$ that form a partition of $\{1,\ldots, L\}$,
  \begin{align*}
    I_1 & \coloneqq \{1\leq j \leq L : ({\mathbf T}\, {\mathbf y})_j > \gamma \},\\
    I_2 & \coloneqq\{1\leq j \leq L : ({\mathbf T}\, {\mathbf y})_j < -\gamma \},\\
    I_3 & \coloneqq\{1\leq j \leq L : ({\mathbf T}\, {\mathbf y})_j \in [-\gamma, \gamma] \}.
  \end{align*}
Suppose that by contrast $\|{\mathbf T}{\mathbf y}\|_\infty > \gamma$, which means that $I_1 \cup I_2 \neq \emptyset$.
  Then
  \begin{equation}\label{stysum}
    S_\gamma {\mathbf T}\, {\mathbf y} = \sum_{j\in I_1}([{\mathbf T}\, {\mathbf y}]_{j} - \gamma) \, {\mathbf e}_{j} + \sum_{j\in I_2}([{\mathbf T}\,{\mathbf y}]_{j} + \gamma) \, {\mathbf e}_{j},
  \end{equation}
  where ${\mathbf e}_{j}$ denotes the $j$-th unit vector  in ${\mathbb R}^{L}$.
  Now we combine \eqref{stysum} with \eqref{tst0}  and use the invertibility of ${\mathbf T}^{*} {\mathbf T}$ to get
  \begin{equation}
{\mathbf 0}  = ({\mathbf T}^{*} {\mathbf T}){\mathbf T}^\dagger S_\gamma {\mathbf T} \, {\mathbf y}
    = \sum_{j\in I_1}([{\mathbf T}\, {\mathbf y}]_{j} - \gamma)\,{\mathbf T}^* {\mathbf e}_{j} + \sum_{j\in I_2}([{\mathbf T}\, {\mathbf y}]_{j} + \gamma)\,{\mathbf T}^* {\mathbf e}_{j}.
    \label{sumte}
  \end{equation}
  In other words, the set $\{{\mathbf T}^{*}{\mathbf e}_{j} : j \in I_1 \cup I_2\}$ is linearly dependent.
  At the same time, none of these vectors ${\mathbf T}^{*}{\mathbf e}_{j}$ vanishes because  ${\mathbf T}^{*}{\mathbf e}_{j} = {\mathbf 0}$ for  $j\in I_1 \cup I_2$ leads to the following contradiction,
$$
    0 =  \left|\langle {\mathbf T}^{*}{\mathbf e}_j,  {\mathbf y}\rangle_{2}\right| = \left|\langle {\mathbf e}_j,  {\mathbf T}\, {\mathbf y}\rangle_{2}\right|
     =\left|[{\mathbf T}{\mathbf y}]_j\right| > \gamma.
$$
Without loss of generality  assume that $I_1\neq\emptyset$ and choose $j_1\in I_1$.
  Then $[{\mathbf T}{\mathbf y}]_{j_1}> \gamma$ and \eqref{sumte} implies
  \begin{equation}\label{vjtilde}
    {\mathbf T}^{*} {\mathbf e}_{j_{1}} = - \sum_{\substack{j\in I_1\\j\neq j_1}} \frac{[{\mathbf T}{\mathbf y}]_j- \gamma}{[{\mathbf T}{\mathbf y}]_{j_1} - \gamma} \, {\mathbf T}^{*} \, {\mathbf e}_{j} - ~\sum_{j\in I_2} \frac{[{\mathbf T}{\mathbf y}]_j + \gamma}{[{\mathbf T}{\mathbf y}]_{j_1} - \gamma}\,{\mathbf T}^{*} \, {\mathbf e}_{j}.
  \end{equation}
A closer look at the coefficients shows that
  \begin{equation*}
    \frac{[{\mathbf T}{\mathbf y}]_j- \gamma}{[{\mathbf T}{\mathbf y}]_{j_1} - \gamma} > 0 \text{~for~} j\in I_1\setminus \{j_1\}
    \qquad\text{and}\qquad
    \frac{[{\mathbf T}{\mathbf y}]_j + \gamma}{[{\mathbf T}{\mathbf y}]_{j_1} - \gamma} < 0 \text{~for~} j\in I_2.
  \end{equation*}
Hence, 
  \begin{align*}
    [{\mathbf T}{\mathbf y}]_{j_1}
    &=\langle {\mathbf T}{\mathbf y}, {\mathbf e}_{j_1}\rangle_{2} = \langle {\mathbf y},  {\mathbf T}^{*} {\mathbf e}_{j_{1}} \rangle_{2}\\
    &=\Big\langle {\mathbf y},\,  \Big( - \sum_{\substack{j\in I_1\\j\neq j_1}} \frac{[{\mathbf T}{\mathbf y}]_j- \gamma}{[{\mathbf T}{\mathbf y}]_{j_1} - \gamma} \, {\mathbf T}^{*} {\mathbf e}_{j} - \sum_{j\in I_2} \frac{[{\mathbf T}{\mathbf y}]_j+ \gamma}{[{\mathbf T}{\mathbf y}]_{j_1} - \gamma}\, {\mathbf T}^{*} {\mathbf e}_{j} \Big)\Big\rangle_{2} \\
    &= - \sum_{\substack{j \in I_1\\j \neq j_1}} \underbrace{\frac{[{\mathbf T}{\mathbf y}]_j- \gamma}{[{\mathbf T}{\mathbf y}]_{j_1} - \gamma}}_{> 0}\,\underbrace{\langle {{\mathbf T} \, {\mathbf y}, {\mathbf e}_j\rangle}_{2}}_{>\gamma}
   -
    \sum_{\substack{j \in I_2}} \underbrace{\frac{[{\mathbf T}{\mathbf y}]_j+ \gamma}{[{\mathbf T}{\mathbf y}]_{j_1} - \gamma}}_{< 0}\,\underbrace{\langle {\mathbf T} \, {\mathbf y}, {\mathbf e}_j\rangle_{2}}_{<-\gamma}
    <0,
  \end{align*}
which contradicts the above assumption that $j_1 \in I_1$.
Thus, $I_1\cup I_2 = \emptyset$, i.e.,  all indices are located within $I_3$, which readily shows that $\|{\mathbf T}{\mathbf y}\|_\infty \leq \gamma$.
\end{proof}
\medskip

Further, if all components of ${\mathbf T} {\mathbf x}$ have a modulus  greater than $\gamma (\| {\mathbf T} {\mathbf T}^{\dagger} \|_{\infty} + 1)$, we show that $H({\mathbf x})$ is single-valued.
Here $\|{\mathbf T}{\mathbf T}^{\dagger}\|_{\infty} \coloneqq \max_{\|{\mathbf z}\|_{\infty}=1}  \|{\mathbf T}{\mathbf T}^{\dagger} {\mathbf z}\|_{\infty}$ denotes the usual  row-sum norm and $\|{\mathbf z}\|_{\infty} \coloneqq \max_{j=1, \ldots , L}|z_{j}|$ for ${\mathbf z}=(z_{j})_{j=1}^{L} \in {\mathbb R}^{L}$.

\begin{theorem}
 Let ${\mathbf T} \in\mathbb{R}^{L\times N}$  with $L \ge N$ with full rank $N$.
  Further let $\gamma > 0$ and  ${\mathcal U}_{{\mathbf T}, \gamma} \coloneqq\{{\mathbf x} \in {\mathbb R}^{N}: \, |({\mathbf T} {\mathbf x})_{j} | > \gamma  (\| {\mathbf T} \, {\mathbf T}^{\dagger} \|_{\infty} + 1) \; \forall \, j=1, \ldots , L\}$.
  Then   $H({\mathbf x})$ in $(\ref{H})$ is single-valued for ${\mathbf x} \in {\mathcal U}_{{\mathbf T}, \gamma}$
and we have  $H({\mathbf x}) = \gamma \, {\mathbf T}^{\dagger} \mathrm{sign} ({\mathbf T} {\mathbf x})$.
\end{theorem}

\begin{proof} 
By Theorem   \ref{Hwell},  $H({\mathbf x})$ in not empty for each ${\mathbf x} \in {\mathbb R}^{N}$, and each fixed point  of
$f_{{\mathbf x},{\mathbf T}}$ in (\ref{fxT}) provides us an element ${\mathbf y} = {\mathbf t} - {\mathbf x} \in H({\mathbf x})$.
We show that  $f_{{\mathbf x}, {\mathbf T}}$ possesses only one fixed point, if ${\mathbf x} \in {\mathcal U}_{{\mathbf T}, \gamma}$.
Assume by contrast that $f_{{\mathbf x},{\mathbf T}}$
possesses two fixed points ${\mathbf t}_{1}$ and ${\mathbf t}_{2}$ with ${\mathbf t}_{1} \neq {\mathbf t}_{2}$ .
Thus
$$ {\mathbf t}_{1} = {\mathbf x} + ({\mathbf I}_{N} - {\mathbf T}^{\dagger} S_{\gamma}{\mathbf T}) {\mathbf t}_{1}, \qquad {\mathbf t}_{2} = {\mathbf x} + ({\mathbf I}_{N} - {\mathbf T}^{\dagger} S_{\gamma}{\mathbf T}) {\mathbf t}_{2}, $$
implies
$ {\mathbf x} = {\mathbf T}^{\dagger} S_{\gamma}{\mathbf T} {\mathbf t}_{1} = {\mathbf T}^{\dagger} S_{\gamma}{\mathbf T} {\mathbf t}_{2}.$
Consequently, since ${\mathbf T}^{*} {\mathbf T}$ is invertible, we have
\begin{equation}\label{t12}
{\mathbf T}^{*} S_{\gamma}{\mathbf T} {\mathbf t}_{1} - {\mathbf T}^{*} S_{\gamma}{\mathbf T} {\mathbf t}_{2} = {\mathbf 0}.
\end{equation}
Observe that  for ${\mathbf x} \in {\mathcal U}_{{\mathbf T}, \gamma}$ it follows that $\textrm{sign}\,  {\mathbf T}\, {\mathbf x} = \textrm{sign}\, {\mathbf T}\, {\mathbf t}_{1}= \textrm{sign} \, {\mathbf T}\,{\mathbf t}_{2}$ and in particular
\begin{equation}\label{gam}
 |[{\mathbf T} {\mathbf t}_{1}]_{j}| >\gamma, \qquad |[{\mathbf T} {\mathbf t}_{2}]_{j}| >\gamma, \qquad \textrm{for all} \quad  j=1, \ldots, L,
 \end{equation}
  since for ${\mathbf t} \in \{{\mathbf t}_{1}, \, {\mathbf t}_{2}\}$,
$$
\|{\mathbf T}{\mathbf x} - {\mathbf T}{\mathbf t}\|_{\infty} = \|{\mathbf T}{\mathbf x} - {\mathbf T}f_{{\mathbf x},{\mathbf T}} ({\mathbf t})\|_{\infty}
= \|{\mathbf T}{\mathbf T}^{\dagger}( {\mathbf T}- S_{\gamma} {\mathbf T}) {\mathbf t} \|_{\infty} \le \|{\mathbf T}{\mathbf T}^{\dagger}\|_{\infty} \, \gamma.
$$
Now
we will show that ${\mathbf T}^{*}(S_{\gamma} {\mathbf T} {\mathbf t}_{1} - S_{\gamma} {\mathbf T} {\mathbf t}_{2}) ={\mathbf 0}$ implies
$S_{\gamma} {\mathbf T} {\mathbf t}_{1} - S_{\gamma} {\mathbf T} {\mathbf t}_{2} = {\mathbf 0}$. By (\ref{gam}) we can then conclude that ${\mathbf T} {\mathbf t}_{1} = {\mathbf T} {\mathbf t}_{2}$  and thus ${\mathbf t}_{1}= {\mathbf t}_{2}$, which leads to the wanted contradiction.
We consider the index sets
\begin{align*}
I_{1} &= \{1 \le j \le L: \, [S_{\gamma} {\mathbf T} {\mathbf t}_{1} - S_{\gamma} {\mathbf T} {\mathbf t}_{2}]_{j} > 0 \}, \\
I_{2} &= \{1 \le j \le L: \, [S_{\gamma} {\mathbf T} {\mathbf t}_{1} - S_{\gamma} {\mathbf T} {\mathbf t}_{2}]_{j} =0 \}, \\
I_{3} &= \{1 \le j \le L: \, [S_{\gamma} {\mathbf T} {\mathbf t}_{1} - S_{\gamma} {\mathbf T} {\mathbf t}_{2}]_{j} <0 \},
\end{align*}
and will show that $I_{1} \cup I_3 = \emptyset$.
Relation (\ref{t12}) implies
$$ {\mathbf 0} = \sum_{j\in I_1}  [S_{\gamma} {\mathbf T} {\mathbf t}_{1} - S_{\gamma} {\mathbf T} {\mathbf t}_{2}]_{j} {\mathbf T}^{*} {\mathbf e}_{j} + \sum_{j\in I_3}  [S_{\gamma} {\mathbf T} {\mathbf t}_{1} - S_{\gamma} {\mathbf T} {\mathbf t}_{2}]_{j} {\mathbf T}^{*} {\mathbf e}_{j}.
$$
Now, suppose contrarily that w.l.o.g.\ $I_{1} \neq \emptyset$, then for $j_{1} \in I_{1}$ we find similarly as in (\ref{vjtilde})
$$ {\mathbf T}^{*} {\mathbf e}_{j_{1}}
= \sum_{\substack{j\in I_1\\j\neq j_1}}  \frac{-[S_{\gamma} {\mathbf T} {\mathbf t}_{1} - S_{\gamma} {\mathbf T} {\mathbf t}_{2}]_{j}}
{[S_{\gamma} {\mathbf T} {\mathbf t}_{1} - S_{\gamma} {\mathbf T} {\mathbf t}_{2}]_{j_{1}}} {\mathbf T}^{*} {\mathbf e}_{j}
+
\sum_{\substack{j\in I_3}}  \frac{-[S_{\gamma} {\mathbf T} {\mathbf t}_{1} - S_{\gamma} {\mathbf T} {\mathbf t}_{2}]_{j}}
{[S_{\gamma} {\mathbf T} {\mathbf t}_{1} - S_{\gamma} {\mathbf T} {\mathbf t}_{2}]_{j_{1}}} {\mathbf T}^{*} {\mathbf e}_{j}. $$
A closer look at the coefficients shows that
\begin{equation*}
  \frac{-[S_{\gamma} {\mathbf T} {\mathbf t}_{1} - S_{\gamma} {\mathbf T} {\mathbf t}_{2}]_{j}}
  {[S_{\gamma} {\mathbf T} {\mathbf t}_{1} - S_{\gamma} {\mathbf T} {\mathbf t}_{2}]_{j_{1}}} < 0 \;\;  \textrm{for} \; j\in I_1\setminus \{j_1\},\qquad
   \frac{-[S_{\gamma} {\mathbf T} {\mathbf t}_{1} - S_{\gamma} {\mathbf T} {\mathbf t}_{2}]_{j}}
  {[S_{\gamma} {\mathbf T} {\mathbf t}_{1} - S_{\gamma} {\mathbf T} {\mathbf t}_{2}]_{j_{1}}} > 0 \;\; \textrm{for} \; j\in I_3.
\end{equation*}
Hence,
  \begin{align*}
    &[{\mathbf T} ({\mathbf t}_{1} - {\mathbf t}_{2})]_{j_{1}}
    = \textstyle \langle {\mathbf T} ({\mathbf t}_{1} - {\mathbf t}_{2}), {\mathbf e}_{j_{1}} \rangle_{2} = \langle {\mathbf t}_{1} - {\mathbf t}_{2}, {\mathbf T}^{*} {\mathbf e}_{j_{1}} \rangle_{2}\\
    &= \textstyle \left\langle\mathbf{t}_1 - \mathbf{t}_2, \sum\limits_{\substack{j\in I_1\\j\neq j_1}}  \frac{-[S_{\gamma} {\mathbf T} {\mathbf t}_{1} - S_{\gamma} {\mathbf T} {\mathbf t}_{2}]_{j}}
    {[S_{\gamma} {\mathbf T} {\mathbf t}_{1} - S_{\gamma} {\mathbf T} {\mathbf t}_{2}]_{j_{1}}} {\mathbf T}^{*} {\mathbf e}_{j}
    +
    \sum\limits_{{j\in I_3}}  \frac{-[S_{\gamma} {\mathbf T} {\mathbf t}_{1} - S_{\gamma} {\mathbf T} {\mathbf t}_{2}]_{j}}
    {[S_{\gamma} {\mathbf T} {\mathbf t}_{1} - S_{\gamma} {\mathbf T} {\mathbf t}_{2}]_{j_{1}}} {\mathbf T}^{*} {\mathbf e}_{j}\right\rangle_{2}\\
    &= \textstyle \sum\limits_{\substack{j\in I_1\\j\neq j_1}}  \underbrace{\textstyle \frac{-[S_{\gamma} {\mathbf T} {\mathbf t}_{1} - S_{\gamma} {\mathbf T} {\mathbf t}_{2}]_{j}}
    {[S_{\gamma} {\mathbf T} {\mathbf t}_{1} - S_{\gamma} {\mathbf T} {\mathbf t}_{2}]_{j_{1}}}}_{<0} \underbrace{\left\langle{\mathbf T} ({\mathbf t}_{1} - {\mathbf t}_{2}), {\mathbf e}_{j}\right\rangle_{2}}_{>0}
    +
    \sum\limits_{{j\in I_3}}  \underbrace{\textstyle \frac{-[S_{\gamma} {\mathbf T} {\mathbf t}_{1} - S_{\gamma} {\mathbf T} {\mathbf t}_{2}]_{j}}
    {[S_{\gamma} {\mathbf T} {\mathbf t}_{1} - S_{\gamma} {\mathbf T} {\mathbf t}_{2}]_{j_{1}}}}_{>0} \underbrace{\left\langle{\mathbf T} ({\mathbf t}_{1} - {\mathbf t}_{2}), {\mathbf e}_{j}\right\rangle_{2}}_{<0}<0,
  \end{align*}
which contradicts the above assumption that $j_1 \in I_1$.
Therefore, $I_{1} \cup I_{3} = \emptyset$,
and for each component it follows  that $[{\mathbf T} {\mathbf t}_{1}]_{j}= [{\mathbf T} {\mathbf t}_{2}]_{j}$.
We conclude that ${\mathbf t}_{1} = {\mathbf t}_{2}$.
Since $H({\mathbf x})$ is single-valued for ${\mathbf x} \in {\mathcal U}_{{\mathbf T}, \gamma}$, it follows with (\ref{gam})
that
$$ {\mathbf x} = {\mathbf T}^{\dagger} S_{\gamma} {\mathbf T} {\mathbf t}_{1} = {\mathbf T}^{\dagger} ({\mathbf T} {\mathbf t}_{1} - \gamma \sign {\mathbf x} ) = {\mathbf t}_{1} - \gamma {\mathbf T}^{\dagger} \sign {\mathbf x}$$
and thus $H({\mathbf x} )= {\mathbf t}_{1} - {\mathbf x} = \gamma {\mathbf T}^{\dagger} \sign {\mathbf x}$.
\end{proof}

\begin{remark} 
Summarizing, the set-valued function $H({\mathbf x})$ in $(\ref{H})$  satisfies
$$ H({\mathbf 0}) = \{{\mathbf y} \in {\mathbb R}^{N}: \, \|{\mathbf T} {\mathbf y}\|_{\infty} \le \gamma \}, \quad \textrm{and} \quad
H({\mathbf x}) = \gamma {\mathbf T}^{\dagger} \sign {\mathbf x} \quad  \textrm{for} \quad  {\mathbf x} \in {\mathcal U}_{{\mathbf T}, \gamma}. $$
For comparison,  the subdifferential of $\gamma \, \|{\mathbf T} \cdot\|_{1}$ is given by
$\gamma \, \partial \| {\mathbf T}{\cdot} \|_{1} ({\mathbf x}) = \gamma \, {\mathbf T}^{*} \, {\sign} \, ({\mathbf T} {\mathbf x})$, see $\cite{Bauschke_2011}$, Corollary $16.42$. Thus, we observe that
$$\gamma \, \partial \| {\mathbf T}{\cdot} \|_{1} ({\mathbf 0}) = H({\mathbf 0})$$
and for ${\mathcal U}_{2\gamma} :=\{{\mathbf x} \in {\mathbb R}^{N}: \, |[{\mathbf T} {\mathbf x}]_{j}| \ge 2\gamma \; \forall \, j=1, \ldots , L\}$
we find $\gamma \, \partial \| {\mathbf T}{\cdot} \|_{1} ({\mathbf x}) = H({\mathbf x})$ if ${\mathbf T}$ is a so-called Parseval frame matrix satisfying ${\mathbf T}^{*} {\mathbf T} = {\mathbf I}_{N}$.
\end{remark}

\section{The frame soft threshold operator is a proximity operator}
\label{sec:4}

Throughout this section, we again assume that ${\mathbf T} \in {\mathbb R}^{L \times N}$  with $L > N$ has full rank $N$, $\gamma >0$ and let $S_{\gamma}$ the soft shrinkage operator given in (\ref{soft}).
In this section, we will show that the set-valued function $H$ in (\ref{H}) is the subdifferential of a proper, lower semi-continuous and convex  function $\Phi$, i.e., $\Phi  \in \Gamma_{0}$. Thus, we will be able to conclude that ${\mathbf T}^{\dagger} S_{\gamma} {\mathbf T}$ is indeed a proximity operator.
Let us first recall the following definition.

\begin{definition}[12.24 in \cite{Rockafellar_1998}]\label{lemma:cyclicalmonotony}
Let $\langle \cdot , \cdot \rangle$ denote  a scalar product in ${\mathbb R}^{N}$.
  A mapping \linebreak
  $H\colon\mathbb{R}^N\rightrightarrows\mathbb{R}^N$ is called \emph{cyclically monotone} if for any $m \in {\mathbb N}$, $m \ge 2$  and any choice of points ${\mathbf x}_1, \ldots, {\mathbf x}_m$ in ${\mathbb R}^{N}$ and elements
  ${\mathbf y}_i \in H({\mathbf x}_i)$ we have
  \begin{equation}\label{cyclicalmonotonyeq}
    \langle  {\mathbf x}_2- {\mathbf x}_1, {\mathbf y}_1\rangle + \langle  {\mathbf x}_3-{\mathbf x}_2,  {\mathbf y}_2\rangle + \cdots \langle  {\mathbf x}_1- {\mathbf x}_m, {\mathbf y}_m\rangle \leq 0 .
  \end{equation}
  We call $H$ \emph{maximally cyclically monotone} if it is cyclically monotone and its graph cannot be enlarged without destroying this property.
\end{definition}

We will employ the following theorem.

\begin{theorem}[\cite{Bauschke_2011}]\label{bau}
A set-valued mapping ${H}:\mathbb{R}^N\rightrightarrows\mathbb{R}^N$
is the subdifferential of  a function  $\Phi \in \Gamma_{0}$,i.e., $H = \partial \Phi$, if and only if $H$ is maximally  cyclically monotone.
\end{theorem}

In order to show, that $H$ in (\ref{H}) is indeed maximally cyclically monotone, we need a
preliminary lemma.

\begin{lemma}\label{mon1}
Let ${\mathbf x}_{1}, \, {\mathbf x}_{2} \in {\mathbb R}^{N}$ and ${\mathbf y}_{1} \in H({\mathbf x}_{1})$, ${\mathbf y}_{2} \in H({\mathbf x}_{2})$.
 Further, let
 $${\mathbf z}_{1} \coloneqq ({\mathbf I}_{L} - S_{\gamma}) {\mathbf T} ({\mathbf x}_{1} + {\mathbf y}_{1}), \qquad
 {\mathbf z}_{2} \coloneqq ({\mathbf I}_{L} - S_{\gamma}) {\mathbf T} ({\mathbf x}_{2} + {\mathbf y}_{2}).$$
 Then
 $$ \langle S_{\gamma} {\mathbf T} ({\mathbf x}_{1} + {\mathbf y}_{1}), {\mathbf z}_{2} -{\mathbf z}_{1} \rangle_{2} \le 0,$$
 where $\langle \cdot , \cdot \rangle_{2}$ denotes  the standard scalar product in ${\mathbb R}^{L}$.
\end{lemma}

 \begin{proof}
From the definition of $S_{\gamma}$ it follows for $x \in {\mathbb R}$
$$ (1-S_{\gamma}) x =  \begin{cases}
x - (x-\gamma) = \gamma  & x > \gamma,\\
x-(x+\gamma) = -\gamma  & x < -\gamma, \\
x & |x|\le \gamma, \end{cases} $$
and therefore for all ${\mathbf x} \in {\mathbb R}^{N}$,
\begin{equation}\label{bound}
| [({\mathbf I}_{L} - S_{\gamma}) {\mathbf T} {\mathbf x}]_{j}| \le \gamma, \qquad  j=1, \ldots , L.
\end{equation}

In particular $\|{\mathbf z}_{1} \|_{\infty} \le \gamma$ and $\|{\mathbf z}_{2} \|_{\infty} \le \gamma$.
Thus, for $[{\mathbf T}({\mathbf x}_{1} + {\mathbf y}_{1})]_{j} > \gamma$ we have $[{\mathbf z}_{1}]_{j} = \gamma$ and $[{\mathbf z}_{2}]_{j} -[{\mathbf z}_{1}]_{j} \le 0$ as well as $ S_{\gamma}[{\mathbf T}({\mathbf x}_{1} + {\mathbf y}_{1})]_{j} > 0$,
while for $[{\mathbf T}({\mathbf x}_{1} + {\mathbf y}_{1})]_{j} < -\gamma$ we have $[{\mathbf z}_{1}]_{j} = -\gamma$ and $[{\mathbf z}_{2}]_{j} -[{\mathbf z}_{1}]_{j} \ge 0$ as well as $ S_{\gamma}[{\mathbf T}({\mathbf x}_{1} + {\mathbf y}_{1})]_{j} < 0$.
Finally, for $|[{\mathbf T}({\mathbf x}_{1} + {\mathbf y}_{1})]_{j}| \le \gamma$ it follows that $S_{\gamma} [{\mathbf T}({\mathbf x}_{1} + {\mathbf y}_{1})]_{j} =0$.
We therefore conclude
$$
\langle S_{\gamma} {\mathbf T} ({\mathbf x}_{1} + {\mathbf y}_{1}), {\mathbf z}_{2} -{\mathbf z}_{1} \rangle_{2} =
\sum_{j=1}^{N} [S_{\gamma} {\mathbf T} ({\mathbf x}_{1} + {\mathbf y}_{1})]_{j} [{\mathbf z}_{2} -{\mathbf z}_{1}]_{j} \le 0. \hfill
$$
and the assertion follows.
\end{proof}

\noindent
Now we can show

\begin{theorem} \label{mon2} Let ${\mathbb R}^{N}$ be equipped with the scalar product $\langle \cdot , \cdot\rangle_{{\mathbf T}}$ as in $(\ref{normneu})$. Then,  ${H}= H_{\gamma}\colon\mathbb{R}^N\rightrightarrows\mathbb{R}^N$ in $(\ref{H})$ is cyclically monotone.
\end{theorem}

\begin{proof}
Let $m \in {\mathbb N}$  and $m \ge 2$. Further, for $i \in \{1, \ldots, m \}$ let ${\mathbf y}_{i} \in H({\mathbf x}_{i})$.
Recall that for ${\mathbf y}_{i} \in H({\mathbf x}_{i})$, we have by  (\ref{H}) that ${\mathbf x}_{i} = {\mathbf T}^{\dagger} S_{\gamma} {\mathbf  T} ({\mathbf x}_{i} + {\mathbf y}_{i})$.
Therefore, ${\mathbf y}_{i}$ can be rewritten as
\begin{equation}\label{y1}
 {\mathbf y}_{i}= ({\mathbf x}_{i} + {\mathbf y}_{i}) - {\mathbf T}^{\dagger} S_{\gamma} {\mathbf T} ({\mathbf x}_{i} + {\mathbf y}_{i}) = {\mathbf T}^{\dagger}({\mathbf I}_{L} -   S_{\gamma}) {\mathbf T} ({\mathbf x}_{i} + {\mathbf y}_{i}).
\end{equation}
Let
\begin{equation}\label{ui}
{\mathbf z}_{i} \coloneqq  ({\mathbf I}_{L} - S_{\gamma}) {\mathbf T}({\mathbf x}_{i} + {\mathbf y}_{i}),
\qquad
 {\mathbf u}_{i} \coloneqq  {\mathbf T} {\mathbf y}_{i} - {\mathbf z}_{i}, \quad i=1, \ldots , m.
\end{equation}
For simplicity we use the convention ${\mathbf x}_{m+1}\coloneqq {\mathbf x}_{1}$ as well as ${\mathbf y}_{m+1}\coloneqq {\mathbf y}_{1}$, and extend that similarly for ${\mathbf z}_{m+1}$, and ${\mathbf u}_{m+1}$.
Then, 
${\mathbf u}_{i} \in {\ker} \,  {\mathbf T}^{\dagger} = {\ker} \,  {\mathbf T}^{*}$, since
\begin{equation}\label{kern} {\mathbf T}^{\dagger} {\mathbf u}_{i} = {\mathbf T}^{\dagger}({\mathbf T} {\mathbf y}_{i} - ({\mathbf I}_{L} - S_{\gamma}) {\mathbf T}({\mathbf x}_{i} + {\mathbf y}_{i}))
= - {\mathbf T}^{\dagger}{\mathbf T} {\mathbf x}_{i} + {\mathbf T}^{\dagger} S_{\gamma} {\mathbf T}({\mathbf x}_{i} + {\mathbf y}_{i}) = {\mathbf 0}.
\end{equation}
%
\noindent
According to (\ref{cyclicalmonotonyeq}) we  have to show  that
$$ A\coloneqq \sum_{i=1}^{m} \langle {\mathbf x}_{i+1}- {\mathbf x}_{i}, {\mathbf y}_{i} \rangle_{\mathbf T} \le 0. $$
We observe that for all $i=1, \ldots , m$,
\begin{equation}
{\mathbf T} {\mathbf x}_{i} + {\mathbf u}_{i} = {\mathbf T} ({\mathbf x}_{i} + {\mathbf y}_{i}) - {\mathbf z}_{i}
= S_{\gamma} {\mathbf T} ({\mathbf x}_{i} + {\mathbf y}_{i}). \label{A.3}
\end{equation}
Using  (\ref{y1})--(\ref{A.3}),  it follows
\begin{align*}
A & = \sum_{i=1}^{m} \langle {\mathbf x}_{i+1}- {\mathbf x}_{i}, \, {\mathbf y}_{i} \rangle_{\mathbf T}
= \sum_{i=1}^{m} \langle {\mathbf T}({\mathbf x}_{i+1}- {\mathbf x}_{i}), \, {\mathbf T}\, {\mathbf y}_{i} \rangle_{2} \\
&= \sum_{i=1}^{m} \langle ({\mathbf x}_{i+1}- {\mathbf x}_{i}), \, {\mathbf T}^{*}{\mathbf T}\, {\mathbf T}^{\dagger} ({\mathbf I}_{L} - S_{\gamma}) {\mathbf T} ({\mathbf x}_{i}+{\mathbf y}_{i}) \rangle_{2}  = \sum_{i=1}^{m} \langle {\mathbf T} ({\mathbf x}_{i+1}- {\mathbf x}_{i}), \, {\mathbf z}_{i} \rangle_{2}\\
& = \sum_{i=1}^{m} \langle ( {\mathbf T} {\mathbf x}_{i+1} + {\mathbf u}_{i+1}) - ( {\mathbf T} {\mathbf x}_{i} + {\mathbf u}_{i}) - {\mathbf u}_{i+1} + {\mathbf u}_{i}, {\mathbf z}_{i} \rangle_{2} \\
& = \sum_{i=1}^{m} \langle S_{\gamma} {\mathbf T} ( {\mathbf x}_{i+1} +  {\mathbf y}_{i+1}), {\mathbf z}_{i} \rangle_{2}
- \sum_{i=1}^{m}  \langle S_{\gamma} {\mathbf T} ( {\mathbf x}_{i} +  {\mathbf y}_{i}), {\mathbf z}_{i} \rangle_{2}
+ \sum_{i=1}^{m}  \langle - {\mathbf u}_{i+1} + {\mathbf u}_{i}, {\mathbf z}_{i} \rangle_{2} \\
&= \sum_{i=1}^{m} \langle S_{\gamma} {\mathbf T} ({\mathbf x}_{i+1} + {\mathbf y}_{i+1}), {\mathbf z}_{i} - {\mathbf z}_{i+1}\rangle_{2} + \sum_{i=1}^{m}  \langle - {\mathbf u}_{i+1} + {\mathbf u}_{i}, {\mathbf z}_{i} \rangle_{2} .
\end{align*}
Bei Lemma \ref{mon1}, the first sum is not positive. Therefore,
\begin{align*}
A &\le  \sum_{i=1}^{m}  \langle - {\mathbf u}_{i+1} + {\mathbf u}_{i}, {\mathbf z}_{i} \rangle_{2} = \sum_{i=1}^{m}  \langle - {\mathbf u}_{i+1} + {\mathbf u}_{i}, {\mathbf T} {\mathbf y}_{i} - {\mathbf u}_{i} \rangle_{2} \\
&= \sum_{i=1}^{m} \langle {\mathbf T}^{*}( - {\mathbf u}_{i+1} + {\mathbf u}_{i}), {\mathbf y}_{i} \rangle_{2} - \sum_{i=1}^{m} \langle {\mathbf u}_{i}, {\mathbf u}_{i} \rangle_{2} + \sum_{i=1}^{m} \langle {\mathbf u}_{i+1}, {\mathbf u}_{i} \rangle_{2}.
\end{align*}
The first sum vanishes, since ${\mathbf u}_{i}$ and ${\mathbf u}_{i+1}$ are in ${\rm  ker} \, {\mathbf T}^{*}$.  Thus
$$
A \le  - \sum_{i=1}^{m} \| {\mathbf u}_{i} \|_{2}^{2} + \sum_{i=1}^{m} \langle {\mathbf u}_{i+1}, {\mathbf u}_{i} \rangle_{2}
\le - \sum_{i=1}^{m} \| {\mathbf u}_{i} \|_{2}^{2} + \frac{1}{2} \left( \sum_{i=1}^{m} \| {\mathbf u}_{i} \|_{2}^{2} + \| {\mathbf u}_{i+1} \|_{2}^{2} \right) =0.
$$
Hence the assertion of the theorem holds.
\end{proof}
\medskip

Now we can conclude the main theorem of this paper.

\begin{theorem}
Let ${\mathbf T} \in {\mathbb R}^{L \times N}$ with $L \ge N$ and full rank $N$.
Then  the operator
$ {\mathbf T}^{\dagger} S_{\gamma} {\mathbf T} $
is a proximity operator of a proper,  lower semi-continuous, convex function $\Phi$ in ${\mathbb R}^{N}$ aligned with the scalar  product $\langle \cdot, \cdot \rangle_{\mathbf T}$.
\end{theorem}

\begin{proof}
First, recall from (\ref{prox1}) that for any vectors $\mathbf{x}$ and $\mathbf{z}$, we have $\mathbf{x} = \prox_\Phi(\mathbf{z})$ if and only if $\mathbf{x} - \mathbf{z} \in \partial\Phi ({\mathbf x})$.
So, by Theorem \ref{bau}, we need to prove  that $H({\mathbf x})$ in (\ref{H}) is maximally cyclically monotone.
As shown in the previous theorem, we already have that $H({\mathbf x})$ is cyclically monotone.
Further, by Theorem \ref{Hwell}, $H({\mathbf x})$ is bounded, i.e., for all ${\mathbf x} \in {\mathbb R}^{N}$ we have that ${\mathbf y} \in H({\mathbf x})$ implies $\|{\mathbf y}\|_{2} \le \gamma \,  \sqrt{L}\, \|{\mathbf T} \|_{2}$.
Therefore, the range of the operator ${\mathbf I}_{N} + H$ is ${\mathbb R}^{N}$.
By Minty's Theorem, see \cite{Bauschke_2011}, Theorem 21.1, it follows that $H({\mathbf x}) $ is also maximally monotone.
The assertion now follows from Theorem \ref{bau}.
\end{proof}

\begin{remark}
1. We have shown that ${\mathbf T}^{\dagger} S_{\gamma} {\mathbf T}$ is a proximity operator, i.e., there is a $\Phi \in \Gamma_{0}$ with $\prox_{\Phi}= {\mathbf T}^{\dagger} S_{\gamma} {\mathbf T}$. We can however not give a closed expression for $\Phi$. Using the notion of infimal convolution, given by
$(f \, \Box \, g)({\mathbf x} ) \coloneqq \inf_{{\mathbf y} \in {\mathbb R}^{N}} (f({\mathbf y}) + g({\mathbf x}- {\mathbf y}))$ for $f, g \in \Gamma_{0}$ with respect to ${\mathbb R}^{N}$,
one can write,
$$ \Phi({\mathbf x}) = \gamma \|\cdot\|_{1} \Box \big( \tfrac{1}{2} \| \cdot\|_{2}^{2} \big)({\mathbf T} {\mathbf x}) $$
as indicated in \cite{HHN20}, Corollary 4.2. Conversely, as outlined before, there is no closed representation of the proximity operator of $\Phi({\mathbf x}) = \| {\mathbf T} {\mathbf x}\|_{1}$ for the considered case of non-surjective matrices ${\mathbf T}$. We have however seen in Section 2, that the proximity operator of $\| {\mathbf T} {\mathbf x}\|_{1}$ behaves similarly as ${\mathbf T}^{\dagger} S_{\gamma} {\mathbf T}$, particularly for Parseval frame matrices satisfying ${\mathbf T}^{\dagger} = {\mathbf T}^{*}$.

2.  Assume that we have ${\mathbf T} \in {\mathbb R}^{N \times L}$ of full rank $N \le L$. Taking a set-valued mapping
${H}= H_{\gamma}\colon\mathbb{R}^N\rightrightarrows\mathbb{R}^N$ that is defined by ${\mathbf y} \in H({\mathbf x})$ if ${\mathbf x} = {\mathbf T}^{\dagger} \, \mathrm{prox}\,  {\mathbf T} ({\mathbf x} + {\mathbf y})$ with a proximity operator $\mathrm{prox}: {\mathbb R}^{L} \to {\mathbb R}^{L}$, then we can show similarly as above that $H$ is maximally cyclically monotone and therefore ${\mathbf T}^{\dagger} \, \mathrm{prox} \, {\mathbf T} (\cdot)$ is a proximity operator in ${\mathbb R}^{N}$ with respect to the norm $\langle \cdot, \cdot \rangle_{\mathbf T}$. Here the change of the norm is crucial and enables us to use the properties of the proximity operator  in ${\mathbb R}^{L}$. We refer to \cite{HHN20} for a different proof of this assertion.
\end{remark}

\section*{Acknowledgements}
The authors are supported by the  German Science Foundation (DFG) in the Collaborative Research Centre ``SFB 755: Nanoscale Photonic Imaging'' and partially in the framework of the Research Training Group ``GRK 2088: Discovering Structure in Complex Data: Statistics meets Optimization and Inverse Problems'' which he both gratefully acknowledges.

The authors thank Stefan Loock who helped to lay the foundation to this manuscript and Russell Luke and Gabriele Steidl for fruitful discussions. Further, the authors thank the reviewer for advices to clarify notations and to improve the representation of the manuscript.

\bibliographystyle{siam}
\begin{small}
\bibliography{__all}
\end{small}

\end{document}